\documentclass[reqno]{amsart}
\usepackage{latexsym}
\usepackage{amssymb,amsthm,amsmath}
\usepackage[dvips]{graphicx}

\usepackage{color}
\usepackage[colorlinks=true]{hyperref}
\hypersetup{urlcolor=blue, citecolor=blue}
\usepackage{tikz}
\usepackage{multicol}

\let\mib=\boldsymbol
\let\cal=\mathcal

\def\N{{\bf N}}
\def\R{{\bf R}} 
\def\Zd{{\bf Z}^d}
\def\ph{\varphi} 
\def\vD{{\sf D}} 
\def\vu{{\sf u}}
\def\mj{{\bf j}} 
\def\mk{{\bf k}} 
\def\ml{{\bf l}} 
\def\mm{{\bf m}} 
\def\mx{{\bf x}} 
\def\mA{{\bf A}} 
\def\mV{{\bf V}} 
\def\mW{{\bf W}} 
\def\mmu{{\mib \mu}}
\def\mxi{{\mib \xi}} 
\def\mpsi{{\mib \psi}}

\def\eps{\varepsilon} 
\def\lA{{\cal A}}

\def\lL{{\cal L}}
\def\lM{{\cal M}}

\def\aess#1{\;(\hbox{\rm a.e.~} #1)}
\def\Cbc#1{{{\rm C}^{\infty}_{c}(#1)}} 
\def\Cc#1{{{\rm C}_c(#1)}} 

\def\CCp#1{{{\rm C}^{#1}}}
\def\Cj#1{{{\rm C}^1(#1)}} 
\def\Cnl#1{{{\rm C}_0(#1)}} 
\def\Cp#1{{{\rm C}(#1)}} 
\def\Hmj#1{{{\rm H}^{-1}(#1)}} 
\def\Hjnl#1{{{\rm H}^{1}_0(#1)}} 
\def\dv{{\sf div\thinspace}}
\def\dscon{\relbar\joinrel\rightharpoonup}

\def\Dup#1#2{\langle#1,#2\rangle}
\def\Dupp#1#2{\Bigl\langle#1,#2\Bigr\rangle}
\def\Lb#1{{{\rm L}^\infty(#1)}} 
\def\Lbl#1{{{\rm L}^{\infty}_{{\rm loc}}(#1)}} 
\def\Ld#1{{{\rm L}^{2}(#1)}}
\def\Ldl#1{{{\rm L}^{2}_{{\rm loc}}(#1)}}
\def\LLb{{{\rm L}^{\infty}}}
\def\LLd{{{\rm L}^2}}
\def\LL#1{{{\rm L}^{#1}}}
\def\Lp#1{{{\rm L}^{p}(#1)}}
\def\mnul{{\bf 0}} 
\def\oi#1#2{\langle#1,#2\rangle}
\def\ozi#1#2{\langle#1,#2]}
\def\pC#1#2{{{\rm C}^{#1}(#2)}} 
\def\pL#1#2{{{\rm L}^{#1}(#2)}}
\def\pLl#1#2{{{\rm L}^{#1}_{{\rm loc}}(#2)}}
\def\Rd{{{\bf R}^{d}}}

\def\Sdmj{{\rm S}^{d-1}}

\def\tr{{\sf tr}} 
\def\vnul{{\sf 0}}
\def\zoi#1#2{[#1,#2\rangle}

\def\Dscon{\relbar\joinrel\dscon} 
\def\DDscon{\relbar\joinrel\Dscon} 
\def\povrhsk#1{\smash{
        \mathop{\;\Dscon\;}\limits^{#1}}}
\def\povrhdsk#1{\smash{
        \mathop{\DDscon}\limits^{#1}}}

\theoremstyle{plain}
\newtheorem{theorem}{Theorem}
\newtheorem{lemma}[theorem]{Lemma}

\theoremstyle{definition}

\theoremstyle{remark}
\newtheorem{remark}[theorem]{Remark}

\def\endmark{\hskip 2em\begin{picture}(8,10)
\put(0,0){$\Box$} \put(2,0){\rule{1.9mm}{0.3mm}}
\put(6.5,0){\rule{0.3mm}{1.9mm}}
\end{picture}
\par}

\begin{document}
%Exploitation of H-measures in expressing limits of non-quadratic terms
\title[Exploring Limit Behaviour of Non-quadratic Terms via H-measures]{Exploring Limit Behaviour of Non-quadratic Terms via H-measures. Application to Small Amplitude Homogenisation. }
\author{ Martin Lazar}
\address{ Martin Lazar, University of Dubrovnik, Department of Electrical Engineering and Computing, \' Cira Cari\'ca 4, 20000  Dubrovnik, Croatia}
\email{  martin.lazar@unidu.hr}
 \date{\today}

\begin{abstract}
%A method is developed for expressing limits of terms involving powers of an arbitrary integer order by means of H-measures. 
A method is developed for analysing asymptotic behaviour  of terms involving  an arbitrary integer order powers of $\LL p$ functions by means of H-measures. 
It is applied to the small amplitude homogenisation problem for a stationary diffusion equation, in which coefficients are assumed to be analytic perturbations of a constant, enabling formul\ae\ for higher order correction terms in a general, non-periodic setting. Explicit expressions in terms of Fourier coefficients are obtained under periodicity assump\-tion. The method enables its generalisation and application to the corresponding non-stationary equation, as well as to some other small amplitude homogenisation problems.
\end{abstract}



\maketitle

\section{Introduction and motivation}

H-measures, as originally introduced a quarter of century ago by L. Tartar \cite{Tar} and (independently) P. G\'erard \cite{Ger} are kind of a microlocal defect tool, measuring deflection of weak from strong $\LLd$ convergence. They explore   quadratic limit behaviour of  bounded $\LLd$ sequences.

A prominent feature of H-measures is  their capability to keep  track of an equation satisfied by functions  generating them. More precisely, if an H-measure is associated to solutions of an equation $Pu_n=0$ (accompanied by a series of initial/boundary conditions), one can take advantage of their basic properties: the localisation and the propagation one. The former constrains the support of the H-measure within the characteristic set of the (pseudo) differential operator $P$, while the latter states that the measure (as well as concentration and oscillation effects) propagates along bicharacteristics of $P$. 

Since their introduction, they have been successfully applied in many mathematical fields - let us here  mention generalisation of compensated compactness results to equations with variable coefficients \cite{Ger,Tar}, applications in the control theory \cite{ DLR, LZ}, the velocity averaging results  \cite{Ger, LMdpde}, as well as 
explicit formul\ae\ and bounds in homogenisation  \cite{Tar, ALjmaa, ALrwa}.  

Original H-measures are restricted to the $\LLd$ framework. This constraint has partially been overcome with the introduction of H-distributions \cite{AM} - a generalisation of the concept to the $\LL p, p \geq 1$ framework.  More precisely, a new tool is constructed with the aim of exploring products of a form $\int u_n v_n$, with a sequence $(u_n)$ being bounded in $\LL p$, while $v_n$ are taken from the corresponding dual $\LL {p'}$

However, the above generalisation is still bound to a study of quadratic terms (with possibly different factors), and does not handle higher order terms, such as cubic ones. The aim of this paper is to investigate possibilities of handling a general $\LL p, p \geq 2$ sequences and to describe, roughly speaking,  a limit behaviour of $\int |u_n|^p $. Here it should be mentioned that actual products we deal with are more complicated, allowing each factor to be accompanied by a pseudodifferential operator acting on it. 

The analysis below is based on a  relatively simple idea:  for a  sequence of functions $u_n$ converging weakly to zero in  $\LL p, p \geq 2$,  the sequence $(u_n^{p/2})$ is bounded in $\LLd$, and determines (up to a subsequence) an H-measure denoted by $\mu_{p/2}$. 
Thus the limit of $\int |u_n|^p=\int u_n^{p/2} \overline{ u_n^{p/2}} $ can be expressed by the measure $\mu_{p/2}$, which in some cases can be calculated explicitly. Of course, one has to be aware that in the absence of strong convergence, the weak limit of $u_n^{p/2}$ in general differs from zero, which requires correction terms entering the expression. 

The obtained results are applied to  a small amplitude homogenisation problem. The main idea of a small amplitude approximation consists of taking a (formal) expansion of a solution to a problem under consideration with respect to a small parameter representing perturbation of the coefficient. Originally introduced by L. Tartar in \cite{Tar} for a stationary diffusion problem, the approach has subsequently been elaborated and applied to more general homogenisation \cite{ALjmaa, AV}, optimal design  \cite{AG, AK} and inverse problems \cite{GM}. In all these papers H-measures are used as the main analytical tool, and the analysis is performed up to the second order expansion, within which the result is obtained by exploring limits of quadratic terms.
Handling of non-quadratic terms, appearing in higher order expansion, is, however, not achievable  by a standard usage of H-measures and requires a different approach. In this paper we try to make a step forward in that direction and to obtain expressions for higher order correction terms.

The paper is organised as follows. In the next section we describe in details the above presented idea, and demonstrate the method of expressing limit  of higher order terms via (original) H-measures associated to an appropriate combination of given $\LL p$ functions.  Application of the method is demonstrated in Section \ref{Application} on the example of small amplitude homogenisation for a stationary diffusion problem, with a particular intention given to a special, but important case of periodic functions.  The paper is closed by some concluding remarks, and by pointing toward some related and open problems. 

\section{Exploring non-quadratic terms through H-measures}
\label{theory}

\subsection{General setting}
The original H-measures explore quadratic limit behaviour of  $\LLd$ terms, and their existence theorem can be formulated as follows (cf. \cite{Ger, Tar}). 
\begin{theorem}
\label{existence}
 Let $(u_n)$ be a sequence 
converging weakly to zero in $\Ld\Rd $. Then,
after passing to a subsequence (not relabeled), there exists a
nonnegative Radon measure $\mu$ on the cospherical bundle $\Rd\times \Sdmj$  such that for every
$\ph_1, \ph_2 \in \Cnl\Rd $ and $\psi \in \Cp \Sdmj$, it holds
\begin{align}
\label{ex1}
\lim\limits_{n\to\infty}
\int\limits_{\R^{d}} {\cal A}_{\psi} (\ph_1 u_n )(\mx) \overline{(\ph_2 u_n)(\mx)} d\mx 
&= \Dup{\mu}{(\ph_1\overline{\ph_2}) \boxtimes\psi} \\
& = \int\limits_{\Rd\times\Sdmj} \ph_1(\mx)\overline{\ph_2}(\mx)\psi(\mxi)
		 \,d\mu(\mx,\mxi) \,. \nonumber
\end{align}
where ${\cal A}_{\psi}$ is the (Fourier) multiplier operator on
$\R^d$ associated to $\psi(\mxi/|\mxi|)$.

\end{theorem}

The measure $\mu$  is called\/ {\rm the H-measure} associated to 
the (sub)sequence $(\vu_{n'})$. In the sequel we shall often abused terminology and notation  by assuming that we have already passed to  a subsequence determining an H-measure. 

{\bf Notation.}
Throughout the paper by $\Dup\cdot\cdot$ stands for a sesquilinear dual product, taken to be antilinear in the first, while linear in the second variable. By  $\boxtimes$ we denote the tensor product of functions in different variables. 
\hfill\endmark
\vskip 2mm

The theorem also generalises to the vector sequences $(\vu_n)$, resulting in a hermitian, positive semi-definite matrix H-measure, whose diagonal elements are (scalar) H-measures associated to a corresponding component $u_n^i$. 

The proof of the last theorem is based on the (First) commutation lemma \cite[Lemma 1.7]{Tar}, enabling exchange of the multiplier operator $A_\psi$ and the operator of multiplication by $\ph_i$ in \eqref{ex1} when passing to the limit. Thus the limit depends on the product $\ph_1\ph_2$ only and results in a bilinear functional on $\Cnl\Rd\otimes \Cp{\Sdmj}$. Taking into account its positivity, and by using the Schwartz kernel theorem, as well as the Schwartz lemma on nonnegative distributions, one shows that the functional is a Radon measure in both variables, $\mx$ and $\mxi$. 

A multiplier operator $A_\psi$ associated to a bounded symbol $\psi$ is a continuous operator on $\Ld\Rd$, which is easily demonstrated by means of the Fourier transform. Generalisation of that result to an $\LL p$ setting is provided by  the Marcinkiewicz multiplier theorem (e.g. \cite[Theorem 5.2.4]{Gra} showing that $A_\psi$  is a bounded operator on $\Lp \Rd$, for any $p\in\oi1\infty$ and $\psi$ of class $\CCp d$.  

We would like to generalise Theorem \ref{existence} by considering higher order expressions in $u_n\in\LL p, p\in \N$. More precisely we are interested in expressing the limit of
\begin{equation}
\label{p-product}
\int_{\R^{d}} {\cal A}_{\psi_1} (\ph_1 u_n )(\mx) {\cal A}_{\psi_2} (\ph_2 u_n )(\mx) \dots {\cal A}_{\psi_p} (\ph_p u_n )(\mx) d\mx ,
\end{equation}
where $\ph_i, \psi_i, i=1..p$ are appropriate test functions.
The following result holds.

\begin{theorem}
\label{general}
Let $(u_n)$ be a  sequence converging weakly to zero  in $\pLl {p+\eps}\Rd$ for some $ p \in \N$ and $\eps>0$. Then for any choice of test functions 
$\ph_i \in \Cc\Rd, \psi_i \in \pC d \Sdmj, i=1..p$ it holds
\begin{equation}
\label{result}
\lim_n \int_{\R^{d}} {\cal A}_{\psi_1} (\ph_1 u_n )(\mx)\cdot \ldots \cdot {\cal A}_{\psi_p} (\ph_p u_n )(\mx) d\mx 
= \Dup{\mu_{vw}}{\ph \boxtimes1} + \int (\ph v) (\mx)  \,\overline w (\mx) d\mx,
\end{equation}
where $\ph=\prod_{i=1}^p \ph_i$, while $\mu_{vw}$ is the off-diagonal component of the matrix H-measure associated to the sequence $(v_n-v, w_n-w)$ defined by \eqref{v} and \eqref{w} below. 
\end{theorem}
\begin{proof}
By the theorem assumptions, the products $\ph_i u_n$ are bounded in $\pL {p+\eps}\Rd$ for all $i=1..p$. Additionally the Marcinkiewicz multiplier theorem  provides boundedness of ${\cal A}_{\psi_i} (\ph_i u_n )$ in the same space, as well. As by the (commutation ) Lemma \ref{commutation} the commutator $C_i = {\cal A}_{\psi_i} \ph_i - \ph_i{\cal A}_{\psi_i}$ is a compact operator on $\pL p \Rd$, we can consecutively
 exchange the order of operators in \eqref{result}, showing that the corresponding limit equals
\begin{equation}
\label{p-linear}
\lim_n \int_{\R^{d}} \ph {\cal A}_{\psi_1} (\chi_1 u_n )(\mx)\cdot \ldots \cdot {\cal A}_{\psi_p} (\chi_p u_n )(\mx) d\mx 
=\lim_n \int_{\R^{d}} \ph v_n \overline{ w_n} d\mx.
\end{equation}
where  the last two  factors are defined by
\begin{equation}
\label{v}
{\displaystyle}
v_n:= \left\{
\begin{matrix}
\prod\limits_{i=1}^{p/2} {\cal A}_{\psi_i}  (\chi_i u_n ), \quad p\in 2\N \\
\prod\limits_{i=1}^{(p-1)/2} {\cal A}_{\psi_i} (\chi_i u_n ) \Big({\cal A}_{\psi_{(p+1)/2}}(\chi_{(p+1)/2} u_n )\Big)^{1/2}, \quad {\rm elsewhere}
\end{matrix}
\right. 
\end{equation}
 and
\begin{equation}
\label{w}
{\displaystyle}
\overline{ w_n}:= \left\{
\begin{matrix}
 \prod\limits_{i=p/2+1}^{p} {\cal A}_{\psi_i} (\chi_i u_n ), \quad p\in 2\N \\
 \Big({\cal A}_{\psi_{(p+1)/2}} (\chi_{(p+1)/2} u_n )\Big)^{1/2} \prod\limits_{i=(p+3)/2}^{p} {\cal A}_{\psi_i} (\chi_i u_n ), \quad {\rm elsewhere.}
\end{matrix}
\right. 
\end{equation}
In the above $\chi_i$ stands for the characteristic function of the support of $\ph_i$. The conjugation sign originate from the scalar product used in the definition of H-measures \eqref{ex1}. We also assume that one branch of the square root has been selected (the choice does not effect the result). 

Sequences $(v_n)$ and $(w_n)$ are bounded in $\Ld \Rd$, and (after possible passing to a subsequence) we denote their weak $\LLd$ limits by $v$ and $w$, respectively. Note that in the absence of the strong (zero) convergence of $(u_n)$, the limits $v$ and $w$ in general differ from zero. 

Rewriting the last term in \eqref{p-linear} as 
$$
\lim_n \int_{\R^{d}} \ph (v_n-v) \overline{( w_n-w)} d\mx + \int_{\R^{d}} \ph v \overline w d\mx,
$$
and expressing the last limit via H-measure $\mu_{vw}$ one obtains the result. 
\end{proof}

The generalisation of the First commutation lemma to the non-$\LLd$ framework can be found in \cite{AM}, whose (slightly corrected)  version we reproduce here.
\begin{lemma}
\label{commutation}
Let $(u_n)$ be a bounded sequence of  functions  in $\Ld\Rd\cap \pL p\Rd$ for some $p\in\ozi 2\infty$, converging weakly to zero (in the sense of distributions). Denote by $C={\cal A}_{\psi} \ph - \ph{\cal A}_{\psi}$ the commutator determined by $\ph \in \Cnl\Rd$ and $\psi \in \pC d \Sdmj$.  Then the sequence $\left(Cu_n\right)$ converges to zero strongly in $\pL q\Rd$ for any $q\in\zoi 2p$. 
\end{lemma}
The proof of the lemma is based on the interpolation inequality, the First commutation lemma (\cite[Lemma 1.7]{Tar}, providing compactness of the commutator $C$ on $\Ld\Rd$), and the Marcinkiewicz multiplier theorem (providing boundedness of the multiplier $\lA_\psi$ on $\pL q\Rd $ for any $q\in \oi 1\infty$ and $\psi$ smooth enough).

\begin{remark}
\begin{itemize}
\item
The additional regularity  assumptions on $(u_n)$ (requiring it to belong to $\pL {p+\eps}\Rd$ for some positive $\eps$) and $\psi$ (assuming it to be of class $\CCp d$) are necessary for the application of the (commutation) Lemma \ref{commutation}. It enables the limit in \eqref{result} to depend on the product $\ph$ of test functions $\ph_i$ only. 
\item Note that by relation \eqref{p-linear} it follows that the limit in \eqref{result} defines a continuous $(p+1)$-linear form $B$ on $\Cnl\Rd\times   \pC d \Sdmj \times \ldots \times \pC d \Sdmj$. In the special case of $p=2$ one can use Plancherel theorem in order to get a bilinear form depending on products $\ph=\prod_i \ph_i$ and  $\psi=\prod_i \psi_i$ only, resulting in a distribution, and eventually an (H-)measure on $\Rd\times \Sdmj$.  Usually this tool is not at our disposal,  and we are left with a more general object. However, relation \eqref{result} shows that the object is a measure with respect to $\mx$ variable. More precisely, for any choice of test functions $\psi_i, i=1..p$, the functional  $B(\cdot, \psi_1, ... ,\psi_p)$ equals the measure $\mu_{vw}  + v \overline w \lambda(\mx)$, with the last term standing for the Lebesgue measure. 
\item The last theorem can be adapted in order to compute limits of more generalised expressions of the form (for simplicity of the notation we restrict to the even $p$):
\begin{align}
\label{gen-product}
 \int_{\R^{d}} &{\cal A}_\psi\Big({\cal A}_{\psi_1} (\ph_1 u_n )\cdot \ldots \cdot {\cal A}_{\psi_{p/2}} (\ph_{p/2} u_n ) \Big) (\mx) \\
&\cdot {\cal A}_{\psi_{p/2+1}} (\ph_{p/2+1} u_n )\cdot \ldots \cdot {\cal A}_{\psi_p} (\ph_{p} u_n )(\mx) d\mx \,.\nonumber
\end{align}
The difference compared to \eqref{result} is in an additional multiplier (with symbol $\psi$) acting on half of the factors.

Carefully repeating steps of the proof of Theorem \ref{general} one obtains
that the limit of \eqref{gen-product} equals
$$
\Dup{\mu_{vw}}{\ph \boxtimes\psi} + \int {\cal A}_\psi\big(\ph_1\cdot \ldots \cdot\ph_{p/2} \ v\big)(\mx) \, \overline{\ph_{p/2+1}\cdot \ldots \cdot\ph_{p} w } (\mx) d\mx\,
$$
the result coinciding with \eqref{result} for $\psi=1$. 

\end{itemize}
\end{remark}

\subsection{Localisation principle}
\label{local}
As already stated in the introduction, one of the most important properties of  H-measures is the so called localisation property (e.g. \cite[Corollary 2.2]{Ger}),  playing a crucial role in most of successful  applications of the tool. Thus it would be important to check if the same principle applies when analysing limit behaviour of higher order terms. Here we substantiate a positive answer for an $\LLb$ sequence, but it holds for more general situations as well. 

Suppose $u_n\in \Lb\Rd$ are solutions to an equation $P u_n=0$, converging vaguely to zero. By the localisation principle, the associated H-measure is supported within  the characteristic set of the  operator $P$. But the same constraint also applies for off-diagonal terms of a matrix H-measure associated to a vector sequence $(u_n, w_n)$, with $(w_n)$ being an arbitrary bounded $\LLd$ sequence. 

For that reason, we can split the integrand in \eqref{p-product} into two parts, the first one coinciding with the first factor, while the second one containing the rest of them, i.e. $w_n= \prod\limits_{i=2}^{p} {\cal A}_{\psi_i} (\ph_i u_n )$.  Taking into account Remark 4 (the last item), we get the limit of \eqref{p-product}  equals $\Dup{\mu_{uw}}{\ph\boxtimes\psi_1}$, where $\mu_{uw}$ is an off-diagonal term of an H-measure associated to  $(u_n, w_n-w)$. In accordance to above, it satisfies the same constraints as an H-measure associated to $(u_n)$, i.e. its support 
is localised within the characteristic set of $P$.

\subsection{Explicit formul\ae\ in the periodic setting}
%\begin{example}{\bf (Periodic functions)}

Let $(u_n)$ be a bounded sequence of periodic functions in $\Lbl\Rd$ defined by
\begin{equation}
\label{periodic}
u_n(\mx)= \sum_{\mk\in\Zd} \hat u_{\mk} \,e^{2\pi i n \mk \cdot \mx}.
\end{equation}
Hereby we assume the mean value  $\hat u_\vnul$ equals zero, thus obtaining the zero weak convergence of $(u_n)$. The associated H-measure is well known, and is combination of the Lebesgue measure (in $\mx$) and the Dirac mass (in $\mxi$):
\begin{equation}
\label{H-periodic}
\mu(\mx, \mxi)=\sum_\mk |\hat u_\mk|^2  \delta_{\mk\over|\mk| }(\mxi)\, \lambda(\mx) \,.
\end{equation}
By using the above described procedure, we would like to express the limit of \eqref{p-product} in the case $p=4$. 
First, let us note that for a periodic function $u_n$, the action of the multiplier ${\cal A}_{\psi}$ results in a periodic function again:
$$
{\cal A}_{\psi} u_n (\mx)= \sum_\mk \hat u_\mk  \,\psi({\mk }) e^{2\pi i n \mk \cdot \mx},
$$
where we have used that the symbol of the above multiplier is a homogeneous function of order zero. 
Specially, function $v_n$ defined by \eqref{v} in this case takes the form
$$
v_n (\mx)= \big({\cal A}_{\psi_1} u_n \big)\big( {\cal A}_{\psi_2} u_n \big) (\mx)
=\sum_{\mj, \mk} \hat u_\mj  \,\hat u_\mk   \,\psi_1({\mj }) \psi_2({\mk }) \,e^{2\pi i n (\mj+ \mk) \cdot \mx},
$$
and converges weakly to
$$
v(\mx)= \sum_{\mk} \hat u_\mk \, \hat u_{-\mk}   \,\psi_1({\mk }) \psi_2(-{\mk })\,.
$$
Similarly,
\begin{align*}
\overline{w_n}(\mx)
&=  \big({\cal A}_{\psi_3} u_n\big)  \big({\cal A}_{\psi_4} u_n \big)  (\mx)
=\sum_{\ml, \mm} \hat u_\ml \,\hat u_\mm  \,\psi_3({\ml }) \psi_4({\mm })\, e^{2\pi i n (\ml+ \mm) \cdot \mx}\\
& \dscon \overline w(\mx)= \sum_{\ml} \hat u_\ml \,\hat u_{-\ml}  \,\psi_3({\ml }) \psi_4(-{\ml })\,.
\end{align*}
The measure $\mu_{vw}$ determined by the sequences $(v_n-v)$ and $(w_n-w)$ then reads
\begin{equation*}
%\label{H-quartic}
\mu_{vw}=\sum_{\begin{matrix}  \scriptstyle \mj, \mk \\  \scriptstyle \mj+\mk\not=\{\vnul\}\end{matrix}} \!
\Big(\!\!\sum_{\begin{matrix}  \scriptstyle \ml, \mm \\  \scriptstyle \ml+\mm=-(\mj+\mk)\end{matrix}}
 \!\!\!\!\hat u_\mj \,\hat u_\mk \,\hat u_\ml \,\hat u_\mm \,
\psi_1({\mj }) \psi_2({\mk })
\psi_3({\ml }) \psi_4({\mm })\Big)\,
\delta_{\mj+\mk \over |\mj+\mk|}(\mxi) \,\lambda(\mx) \,.
\end{equation*}
Taking into account the form of the limits $v$ and $w$, relation \eqref{result} implies
\begin{align}
\label{result-periodic4}
\lim_n &\int_{\R^{d}} {\cal A}_{\psi_1} (\ph_1 u_n )(\mx)\cdot \ldots \cdot {\cal A}_{\psi_4} (\ph_4 u_n )(\mx) d\mx \\
&= \sum_{\begin{matrix}  \scriptstyle \mj, \mk,  \ml, \mm \\  
\scriptstyle \ml+\mm=-(\mj+\mk)\end{matrix}}
 \!\!\hat u_\mj \,\hat u_\mk \,\hat u_\ml \,\hat u_\mm \,
\psi_1({\mj }) \psi_2({\mk })
\psi_3({\ml }) \psi_4({\mm })\,
\int_{\R^{d}}\ph (\mx) d\mx,\nonumber
\end{align}
where, let it be repeated, $\ph=\prod_i \ph_i$. 

Similar approach can be applied for other values of $p$. Here we briefly comment the case  $p=3$, as a prototype for odd values of $p$. 

Note that due to additional regularity of sequence $(u_n)$ (bounded in $\Lbl\Rd$, and thus in $\pLl q\Rd$, for any $q$), we can avoid splitting of the middle factor in \eqref{p-product} into two terms involving the square root, and define 
$$
v_n={\cal A}_{\psi_1} u_n ,\quad 
\overline{ w_n}=  \big({\cal A}_{\psi_2} u_n\big)  \big({\cal A}_{\psi_3} u_n  \big) ,
$$
which both belong to $\Ldl\Rd$. Thus the associated measure is well defined and computing as above one gets
\begin{equation}
\label{H-cubic}
\mu_{vw}=\sum_{\begin{matrix}  \scriptstyle  \mk \not=\{\vnul\}\end{matrix}} \!
\Big(\!\!\sum_{\begin{matrix}  \scriptstyle \ml, \mm \\  \scriptstyle \mk+\ml+\mm=\vnul\end{matrix}}
 \!\!\!\!\,\hat u_\mk \,\hat u_\ml \,\hat u_\mm \,
\psi_1({\mk })
\psi_2({\ml }) \psi_3({\mm })\Big)\,
\delta_{\mk \over |\mk|}(\mxi) \,\lambda(\mx) \,.
\end{equation}
and
\begin{align}
\label{result-periodic3}
\lim_n &\int_{\R^{d}} {\cal A}_{\psi_1} (\ph_1 u_n )(\mx)\cdot \ldots \cdot {\cal A}_{\psi_3} (\ph_3 u_n )(\mx) d\mx \\
&= \sum_{\begin{matrix}  \scriptstyle \mk,  \ml, \mm \\  
\scriptstyle \mk+\ml+\mm=\vnul\end{matrix}}
 \!\!\hat u_\mk \,\hat u_\ml \,\hat u_\mm \,
\psi_1({\mk }) \psi_2({\ml })
\psi_3({\mm }) \,
\int_{\R^{d}}\ph (\mx) d\mx,\nonumber
\end{align}
Similarly, the last formula is now easily extend to an arbitrary integer $p\geq2$,  providing the following result. 

\begin{theorem}
\label{general-periodic}
Let $(u_n)$ be  a bounded sequence of periodic functions in $\Lbl\Rd$ defined by \eqref{periodic}, with the mean value zero. Then for any $ p \in \N$ and any choice of test functions 
$\ph_i \in \Cc\Rd, \psi_i \in \pC d \Sdmj, i=1..p$ it holds
\begin{equation}
\label{result-periodic}
\lim_n \int_{\R^{d}} {\cal A}_{\psi_1} (\ph_1 u_n )(\mx)\cdot \ldots \cdot {\cal A}_{\psi_p} (\ph_p u_n )(\mx) d\mx 
= \sum_{\begin{matrix}  \scriptstyle \mk_i\in\Zd, \\  
\scriptstyle \sum_i\mk_i=\vnul\end{matrix}}
\Big(\prod_{i=1}^p \hat u_{\mk_i} \psi_i({\mk_i })\Big)\int_{\R^{d}}\ph (\mx) d\mx,
\end{equation}
where $\ph=\prod_{i=1}^p \ph_i$.
\end{theorem}
\begin{remark}
The last theorem easily generalises to a case when each factor of the integrand in \eqref{result-periodic} is associated to a different  sequence $(u_n^i)_n, i=1..p$, just by adjusting the Fourier coefficients on the right hand side. 

Note that the last theorem also incorporates the expression \eqref{H-periodic} for an H-measure associated to a sequence of periodic functions. 
\end{remark}

\section{Computation of higher order correction terms in small amplitude homogenisation}
\label{Application}
\subsection{Small amplitude homogenisation}

We consider a sequence of elliptic problems:
\begin{equation}
\label{e-problems}
\left\{
\begin{aligned}	
	 - \dv(\mA^n\nabla u^n) & = f\in\Hmj\Omega	\\
			u^n   & \in \Hjnl\Omega\,,\cr
\end{aligned}	
\right.
\end{equation}
 where $\Omega\subseteq\Rd$ is an open,  bounded domain.

If the coefficients $\mA^n$ are taken from the set (with $0<\alpha<\beta$)
$$
\lM(\alpha,\beta;\Omega) :=\{\mA\in\Lb{\Omega; \lL(\Rd;\Rd)} : \mA(\mx)\mxi\cdot\mxi\geq\alpha|\mxi|^2
	\, \& \, \mA^{-1}(\mx)\mxi\cdot\mxi\geq{1\over\beta}|\mxi|^2 \aess{\mx\in \Omega} \}\;,
$$
we have that (after  extracting a non-relabelled subsequence) $\mA^n\povrhdsk H \mA^\infty\in\lM(\alpha,\beta;\Omega)$ (e.g. \cite[Theorem 6.5]{Thom}). The above H-convergence means that for any choice of 
$f\in\Hmj\Omega$, the sequence of solutions to \eqref{e-problems}  satisfies:
\begin{align*}
u^n	& \dscon u^\infty	 \qquad\hbox{in } \Hjnl\Omega \\
\mA^n \nabla u^n & \dscon \mA^\infty\nabla u^\infty \quad\hbox{in } \Ld \Omega \;, 
\end{align*}
where $u^\infty$ is the solution of \eqref{e-problems} with $\infty$ instead of $n$.

For small amplitude homogenisation we consider $\mA^n$ to be a perturbation of a constant:
\begin{equation}
\label{expansion} 
\mA^n_\gamma(t,\mx) = \mA_0 + \gamma \mA_1^n(t,\mx) + \gamma^2 \mA_2^n(t,\mx) + \gamma^3 \mA_3^n(t,\mx)+o(\gamma^3) \;,
\end{equation}
where $\mA_i^n\povrhsk\ast\mnul$ in $\Lb \Omega$ for any $i\geq1$.

Assuming that $\mA_0 \in \lM(\alpha,\beta;\Omega)$, we have that (for small values of $\gamma$)
\begin{equation}
\label{A_infinite}
\mA^n_\gamma \povrhdsk H \mA_\gamma^\infty = \mA_0 + \gamma \mA_1^\infty(t,\mx) + \gamma^2 \mA_2^\infty(t,\mx) + \gamma^3 \mA_3^\infty(t,\mx)+o(\gamma^3) \;,
\end{equation}
where the homogenised limit $\mA^\infty_\gamma$ is measurable in $\mx$ and analytic in $\gamma$ (for details on small amplitude homogenisation consult \cite[Chapter 29]{Thom}).

The components of the limit $\mA^\infty_\gamma$ up to the second order of $\gamma$ are given explicitly by means of  H-measures. More precisely, the following theorem holds (cf. \cite[Theorem 4.2]{Tar}).

\begin{theorem}
\label{A2-theorem}
The effective diffusion tensor $\mA^\infty_\gamma$ satisfies
\begin{equation*}
\mA_\gamma^\infty = \mA_0 +   \gamma^2 \mA_2^\infty(t,\mx) + o(\gamma^2),
\end{equation*}
where the second order correction is given by
$$
\int_\Omega \big(\mA_2^\infty\big)_{ij}( \mx) \phi ( \mx)  d\mx 
	= - \sum_{k,l} \Dupp{\mu_{11}^{iklj}}{\phi {\xi_k  \xi_l  
		\over \mA_0\mxi\cdot\mxi}},
$$
with $\mmu_{11}$ standing for an H-measure (with four indices) associated to $\mA_1^n$.
\end{theorem}

\subsection{Expressing higher order correction terms}
The goal of the following study is to use the approach presented in Section \ref{theory} in order to construct explicit expressions for  higher order homogenised terms. We shall see that such obtained formul\ae\ will include H-measures associated to (the powers of) the lower order perturbations. Here we follow the approach presented firstly in \cite{Tar} for elliptic problems (for the parabolic version see \cite{ALjmaa}). 

%Fix a function  $u\in\Hjnl\Omega$, and define 
%$f_\gamma:=- \dv(\mA^\infty_\gamma\nabla u)$. 
%Denote by $$ f_i=-\dv(\mA^\infty_i\nabla u) $$ the corresponding  components  in the expansion of $f_\gamma$ in powers of $\gamma$. 
%Next, denote by $u^n_\gamma$ the solution of \eqref{e-problems} with $\mA^n=\mA^n_\gamma$ and $f=f_\gamma$. 

Fix a function  $u\in\Hjnl\Omega$, and denote by $u^n_\gamma$ the solution of
\begin{equation}
\label{eq-gamma}
-\dv(\mA^n_\gamma\nabla u^n_\gamma) = - \dv(\mA^\infty_\gamma\nabla u).
\end{equation}
Because of H-convergence, we have that $ u^n_\gamma \dscon  u$ (in $\Hjnl \Omega$)
and that $\vD^n_\gamma:=\mA^n_\gamma\nabla u^n_\gamma  \dscon \mA^\infty_\gamma\nabla u$ (in $\Ld \Omega$).
After writing the expansions in powers of $\gamma$:
$$
u^n_\gamma=u^n_0+\gamma u^n_1 + \gamma^2u^n_2 + o(\gamma^2)\;, \quad
\quad	\vD^n_\gamma=\vD^n_0+\gamma \vD^n_1 + \gamma^2\vD^n_2 + o(\gamma^2) \;,
$$
we see that $u^n_0 \dscon u$ and $u^n_i \dscon 0$ for $i\geq1$, while it remains to calculate the limits of $\vD^n_i$. 

 Equating the terms with equal powers of $\gamma$ we get $\nabla u^n_0=\nabla u$, $\vD^n_0=\mA_0\nabla u$ (with $\gamma^0$);
and $\vD^n_1=\mA_0\nabla u^n_1+\mA_1^n\nabla u \dscon\vnul$ (with $\gamma^1$).
On the other hand, $\vD^n_1$ converges to $\mA_1^\infty\nabla u$ (the term in expansion of $\mA^\infty_\gamma \nabla u$ with power of $\gamma$ equal 1),
 which, by varying $u$,  gives that $\mA_1^\infty=\mnul$.

For the quadratic term we have:
\begin{equation}
\label{A2}
\vD^n_2=\mA_0\nabla u^n_2 + \mA_1^n\nabla u^n_1 + \mA_2^n\nabla u \dscon \lim_n \mA_1^n\nabla u^n_1 = \mA_2^\infty \nabla u\;,
\end{equation}
By expansion of \eqref{eq-gamma} and taking terms with $\gamma^1$, observe that $u^n_1$ satisfies the problem \eqref{e-problems} with $\mA^n=\mA_0$ and the right hand side equal to $\dv (\mA_1^n \nabla u)$. As $\mA_0$ is a constant matrix, application of the Fourier transform yields 
\begin{equation*}
%\label{E1}
\widehat{\nabla u_1^n} ( \mxi)
= - { \left(\mxi \otimes \mxi\right) \widehat{\left( \mA_1^n \nabla u\right)} ( \mxi) 
	\over\mA_0\mxi\cdot\mxi},
\end{equation*}
or equivalently
\begin{equation}
\label{E1}
\nabla u_1^n ( \mx)
= - \lA_\Psi \left( \mA_1^n \nabla u\right) ( \mx) ,
\end{equation}
where $\lA_\mpsi $ is the multiplier with the symbol $\Psi(\mxi)={\mxi \otimes \mxi	\over\mA_0\mxi\cdot\mxi}$.

Taking into account \eqref{A2}, one gets that the second order correction $\mA_2$ is expressed via limit of the quadratic term in $\mA_1^n$, i.e. via a corresponding H-measure, which is essentially the  result of  Theorem \ref{A2-theorem}.

In order to find the cubic term, note that
\begin{equation}
\label{A3}
\vD^n_3=\mA_0\nabla u^n_3 + \mA_1^n\nabla u^n_2 + \mA_2^n\nabla u^n_1  + \mA_3^n\nabla u \dscon \lim_n (\mA_1^n\nabla u^n_2+\mA_2^n\nabla u^n_1 ) = \mA_3^\infty \nabla u\;,
\end{equation}
where we have taken into account  zero weak convergence of $(\mA_i^n)$ and $(u_i^n)$ for $i\geq1$. 

On the other hand, by equating the terms  with $\gamma^2$ in the expansion    of the relation \eqref{eq-gamma}, one gets for $u^n_2$
\begin{equation*}
-\dv(\mA_0\nabla u^n_2) =\dv(\mA_2^n\nabla u +  \mA_1^n\nabla u^n_1- \mA_2^\infty\nabla u) \dscon 0,
\end{equation*}
where the convergence of the right hand side follows from \eqref{A2}. 

By applying the Fourier transform, similarly as in  \eqref{E1} one gets
$$
\nabla u_2^n ( \mx)
= - \lA_\Psi \left( \mA_2^n\nabla u +  \mA_1^n\nabla u^n_1- \mA_2^\infty\nabla u\right) ( \mx).
$$
Putting the last expression together with \eqref{E1} in \eqref{A3}, integration and multiplication by a test function $\ph\in\Cbc\Omega$ yields
\begin{equation}
\label{A3-int}
\int_\Omega \ph \mA_3^\infty \nabla u\,  d\mx 
=\lim_n\int_\Omega \ph \Big( - \mA_1^n \lA_\Psi \mA_2^n\nabla u - \mA_2^n\lA_\Psi \mA_1^n\nabla u 
+\mA_1^n\lA_\Psi \left( \mA_1^n \lA_\Psi \mA_1^n\nabla u\right)  \Big)d\mx 
%\Skp{ \lA_\Psi \mA_2^n\nabla u }{\overline{\mA_1^n}} + \Skp{ \lA_\Psi \mA_1^n\nabla u }{\overline{\mA_2^n}}+ \lim_n\Skp{ \lA_\Psi \left( \mA_1^n \lA_\Psi \mA_1^n\nabla u\right) }{\overline{\mA_1^n}}
\end{equation}
The limits of the first two terms on the right hand side are expressed via H-measures determined by functions $\mA_1^n$ and $\mA_2^n$.
The last term in \eqref{A3-int} is a cubic term in $\mA_1^n$. As applications of H-measures so far have been constrained to quadratic expressions, for that reason no expression for the third order correction has been obtained.   In the sequel we shall overcome the restriction and calculate its limit  by means of Theorem \ref{general}.

We analyse the quadratic terms first. Carefully expanding the product by components we obtain
$$
\lim_n\int_\Omega\ph \Big(   \mA_1^n \lA_\Psi \mA_2^n + \mA_2^n\lA_\Psi \mA_1^n\Big)\nabla u  d\mx =\Dup{2\Re \mmu_{12}}{\ph { \mxi \otimes \mxi  \otimes\nabla u
	\over  \mA_0\mxi\cdot\mxi}}\,,
$$
where $\mmu_{12}$ is an off-diagonal block term of the H-measure associated to $(\mA_1^n,  \mA_2^n)$ --  a measure with four indices (the first of them not being contracted above). 

As for the cubic term, observe that by using an pseudodifferential calculus identity $\int \left(\lA_\psi u\right) v\, dx= \int u \,\lA_{\tilde\psi}  v \,dx$  (where $\tilde \psi$ denotes the change of the argument sign, i.e. $\tilde\psi(\xi)=\psi(- \xi)$) it can be rewritten as 
\begin{equation}
\label{cubic}
\lim_n\int_\Omega  \left(\lA_\Psi\left(\ph \mA_1^n\right)^\top\right)^\top  \mA_1^n \lA_\Psi \mA_1^n\nabla u d\mx \,,
\end{equation}
where symmetry and evenness of $\Psi$ has been taken into account. 
Denoting 
\begin{equation}
\label{VW}
 \mV^n= \left(\lA_\Psi\left(\ph \mA_1^n\right)^\top\right)^\top, \quad \mW^n= \mA_1^n \lA_\Psi \mA_1^n,
\end{equation}
Theorem \ref{general} gives that the $i$-th component  of \eqref{cubic} equals $\sum_{j,k} \Dup{\mu_{v^{ik}w^{kj}}}{\ph\partial_j u\otimes1}$ ($\mu_{v^{ik}w^{kj}}$ denotes appropriate component of the four-index H-measure $\mmu_{VW}$ associated to $(\mV^n, \overline \mW^n)$).  Note that we omit the last term appearing in \eqref{result}, as the (zero) weak convergence assumption on coefficients $\mA_1^n$ implies the  weak limit of $\mV^n$ equals zero.
By varying function $u\in\Cj \Omega$ (e.g.~choosing $\nabla u$ constant on $ \omega$, where $\omega\Subset\Omega$) we finally obtain the following result.
\begin{theorem}
\label{high-result}
The third order correction of the effective diffusion tensor $\mA_\gamma^\infty$ defined in \eqref{A_infinite} is given by
\begin{equation}
\label{result-A3}
\int_\Omega \left(\mA_3^\infty \right)^{ij}\ph \,  d\mx 
= - \Dup{2\Re \mmu_{12}^{ij}}{\ph { \mxi \otimes \mxi  
	\over  \mA_0\mxi\cdot\mxi}}
+\Dup{\tr \mmu_{VW}^{ij}}{\ph  \boxtimes 1}\,,
\end{equation}
where $\mmu_{12}^{ij}$ denotes the matrix H-measure  determined by  sequences $(\mA_1^n)$  and $(\mA_2^n)$ with components $\left(\mmu_{12}^{ij}\right)^{kl}=\mu_{12}^{iklj}$,  and similarly for the H-measure $\mmu_{VW}^{ij}$, whose generating sequences are given by \eqref{VW}. 

\end{theorem}
\begin{remark}
Although not given explicitly, the last term in \eqref{result-A3} depends on $\Psi={ \mxi \otimes \mxi  
	\over  \mA_0\mxi\cdot\mxi}$ indeed, as both sequences $(V^n), (W^n)$, as well as the associated H-measure depend on it. This dependence is more apparent if a periodic setting is considered, which is the subject of the next subsection. 
\hfill\endmark
\end{remark}
Similar procedure as in the above can be applied in order to express further order correction terms of the effective diffusion tensor. Derivation for the fourth order term is sketched in the next subsection under periodicity assumption. 
%Here we also shortly present the result for the $\mA_4^\infty$ term. 

\subsection{Periodic setting}
This subsection is devoted to a special, but important case of periodic coefficients in \eqref{e-problems}. By employing techniques of Section \ref{theory} we derive explicit formul\ae\ for higher order terms expressed by Fourier coefficients of $\mA^n$. 

We suppose the coefficients in \eqref{expansion} to be of the form
$$
\mA^n_i(n \mx)=\mA_i(n \mx)=\sum_{\mk\in\Zd}  \hat\mA_{i,\mk} e^{2\pi i n \mk \cdot \mx}, \quad i\in \N
$$
where (just  to simplify calculations) each $\mA_i$ is symmetric with the zero mean value. 

The measure associated to $(\mA_1^n,  \mA_2^n)$ takes the form
$$
\mmu_{12}=\sum_{\mk} \hat\mA_{1,\mk} \otimes \hat\mA_{2,-\mk} \delta_{\mk\over|\mk| }(\mxi) \lambda(\mx)\,.
$$
Thus the first term on the right hand side of \eqref{result-A3} equals
$$
- 2\sum_\mk  {1 \over \mA_0 \mk\cdot\mk} \int_\Omega  (\hat\mA_{1,\mk}  \mk) \otimes  (\hat\mA_{2, -\mk}  \mk ) \, \ph(\mx) d\mx\,.
$$
As for the second term one uses the formula \eqref{H-cubic} for an H-measure associated to cubic terms (in $\mA_1^n$ here). Thus the trace of $\mmu_{VW}$ (the trace being taken with respect to the middle two components) equals 
$$
\tr \mmu_{VW} =
\sum_{\begin{matrix}  \scriptstyle \mk,  \ml, \mm \\  
\scriptstyle \mk+\ml+\mm=\vnul\end{matrix}}
 {\hat\mA_{1,\ml}  \mk \cdot\mm \over (\mA_0 \mk\cdot\mk) (\mA_0 \mm\cdot\mm)} 
(\hat\mA_{1,\mk}  \mk) \otimes  (\hat\mA_{1, \mm}  \mm )\delta_{\mk\over|\mk| }(\mxi) \lambda(\mx),
$$
and for the third order correction we obtain the explicit expression
$$
\mA_3^\infty
=\sum_{\mk\in\Zd} {1 \over \mA_0 \mk\cdot\mk} (\hat\mA_{1,\mk}  \mk) \otimes 
\Big(- 2 \hat\mA_{2, -\mk}  \mk  + \sum_{\begin{matrix}  \scriptstyle   \ml, \mm \in\Zd\\  
\scriptstyle \mk+\ml+\mm=\vnul\end{matrix}}{\hat\mA_{1,\ml}  \mk \cdot\mm \over  \mA_0 \mm\cdot\mm} \hat\mA_{1, \mm}  \mm \Big).
$$
Just for comparison, here we provide an analogous expression for $\mA_2^\infty$ obtained by means of Theorem \ref{A2-theorem}:
$$
\mA_2^\infty
=- \sum_{\mk\in\Zd} {1 \over \mA_0 \mk\cdot\mk} (\hat\mA_{1,\mk}  \mk) \otimes 
\hat\mA_{1, -\mk}  \mk \,.
$$
By virtue of Theorem \ref{general-periodic} derivation of expressions for higher order terms goes similarly, just involving more lengthy computations of linear algebra.  Here we briefly provide the result for $\mA_4^\infty$. 

Similarly as \eqref{A3-int} one obtains
\begin{align*}
\int_\Omega \mA_4^\infty \nabla u d\mx&=
\lim_n \int_\Omega \Big(- \mA_3^n \lA_\Psi \mA_1^n\nabla u - \mA_2^n \lA_\Psi \mA_2^n\nabla u
-\mA_1^n \lA_\Psi \mA_3^n\nabla u \\
&+\left(\lA_\Psi \mA_1^n\right)^\top  \mA_1^n \lA_\Psi \mA_2^n\nabla u
+\left(\lA_\Psi \mA_1^n\right)^\top  \mA_2^n \lA_\Psi \mA_1^n\nabla u\\
&+\left(\lA_\Psi \mA_2^n\right)^\top  \mA_1^n \lA_\Psi \mA_1^n\nabla u
- \left(\lA_\Psi \mW^n\right)^\top \mW^n \nabla u\Big) d\mx\,
\end{align*}
where $\mW^n$ is given by \eqref{VW}.  

The first three terms on the right hand side are quadratic terms whose limit is expressed via standard H-measures. The next three are cubic ones, and the limit is obtained by virtue of relation \eqref{result-periodic3}, similarly as it was done above for $\mA^\infty_3$. The last one is a quartic term in $\mA_1^n$ (as $\mW^n$ are quadratic expressions of $\mA_1^n$), which is treated by means of formula \eqref{result-periodic4}. Performing some tedious, but mostly elementary computations one gets

\begin{align*}
\mA_4^\infty
&=\sum_{\mk\in\Zd}  {1 \over \mA_0 \mk\cdot\mk}  \Bigg(
-2 (\hat\mA_{1,\mk}  \mk) \otimes 
  \hat\mA_{3, -\mk}  \mk  - (\hat\mA_{2,\mk}  \mk) \otimes 
 (\hat\mA_2, -\mk)\\
&+ \sum_{\begin{matrix}  \scriptstyle   \ml, \mm \in\Zd\\  
\scriptstyle \mk+\ml+\mm=\vnul\end{matrix}}
{1 \over  \mA_0 \mm\cdot\mm}
\bigg(\hat\mA_{1,\ml}  \mk \cdot\mm \,
(\hat\mA_{1,\mk}  \mk) \otimes  (\hat\mA_{2, \mm}  \mm )
+\hat\mA_{2,\ml}  \mk \cdot\mm \,
(\hat\mA_{1,\mk}  \mk) \otimes  (\hat\mA_{1, \mm}  \mm )\\
&\hskip 4cm
+\hat\mA_{1,\ml}  \mk \cdot\mm \,
(\hat\mA_{2,\mk}  \mk) \otimes  (\hat\mA_{1, \mm}  \mm )\\
&- \sum_{\begin{matrix}  \scriptstyle   \mj\in\Zd\\  
\scriptstyle \ml+\mm=-(\mj+\mk)\end{matrix}}
{1 \over  \mA_0 (\mj+\mk)\cdot(\mj+\mk)}
(\hat\mA_{1,\mj}  \mk +  \hat\mA_{1,\ml}  \mm) \cdot(\mj+\mk)
(\hat\mA_{1,\mk}  \mk) \otimes  (\hat\mA_{1, \mm}  \mm )\bigg)\Bigg).
\end{align*}

\begin{remark}
  \begin{itemize}
\item   The above results are easily generalised to the case when $\\\mA_i^n \povrhsk\ast\mA_i \not=\mnul$, with correction terms entering  expression for the homogenised limit $\mA^\infty_\gamma$.
\item  In analysis of small amplitude homogenisation  we considered diffusion coefficients $\mA_n$ to be perturbations of a constant. This was important in order to explicitly express $\nabla u_i^n$ via a given function $u$ (relation \eqref{E1} etc). However, one can try to generalise result to a variable $\mA_0$ by justifying and following suggestions given in  \cite{AV} or \cite{Tar}, partially based on the localisation principle for H-measures. Here it is important that the measures analysed in previous section obey the same principle in accordance to subsection \ref{local}, making potential generalisation feasible. 
  \end{itemize}
\end{remark}

\section{Conclusion}

In the paper we have developed a method for expressing limits of non-quadratic terms by means of original H-measures, followed by application to  the small amplitude homogenisation problem for a stationary diffusion equation.
The method provides higher order correction terms of the effective diffusion tensor expressed by H-measures associated to non-quadratic terms in a general, non-periodic setting. Explicit formul\ae\  in terms of Fourier coefficients are obtained under periodicity assumption.

A similar approach can be also conducted in the case of non-stationary diffusion problems by means of parabolic H-measures \cite{ALjmaa,ALjfa}. The latter represent a generalisation of the tool  that  takes into account the difference between the time and space variables, intrinsic to parabolic type problems. 

The corresponding analysis, as well as the explicit formula for the second order homogenised term has been provided firstly by \cite{ALjmaa} for the periodic functions, while non-periodic generalisations   are given in \cite{AV}. 
Results and the approach presented in Section \ref{theory} of expressing  non-quadratic  limits via microlocal defect measures  extend easily to parabolic H-measures  as well. Thus the third order homogenised terms for a non-stationary diffusion problem can  be easily obtained by slightly adapting the proof of Theorem \ref{high-result}, and similarly for higher orders. 

The  presented method for expressing non-quadratic terms can also be ge\-neralised to recently introduced multiscale H-measures \cite{Tmul, AEL}. Unlike the original ones, multiscale variant are capable of distinguishing sequences with different frequencies. Thus, their application in expressing higher order terms in small amplitude homogenisation makes full  sense in the case of coefficients $\mA_i^n$ oscillating on various scales.

The method developed in this paper paves the path for more precise analysis of other problems involving small amplitude homogenisation \cite{AG, AK,GM} in which the analysis so far has been conducted  by means of H-measures
up to the second order expansion. It is important to notice that in all these papers sequences describing the limit terms posses higher regularity than the one required by the original definition of H-measures ($\LLb$ instead of merely $\LLd$ one) and fit within the setting of   the presented approach,  thus making its potential applications feasible.

\section*{Acknowledgements}
 This  work was  supported in part by the Croatian Science Foundation under Grant 9780 WeConMApp.

\section*{Disclosure statement}
The author declares that he has no conflict of interest.

\end{document}